\documentclass[draft]{amsart} 
\usepackage{amssymb,latexsym,amscd, amsthm, epsfig,amsmath,amssymb}
\usepackage{color}

\usepackage{pgf,tikz}
\usetikzlibrary{arrows}
 \newtheorem{thm}{Theorem}[section]
 
 \newtheorem{lem}[thm]{Lemma}
 \newtheorem{prop}[thm]{Proposition}
 \theoremstyle{definition}
 
 \theoremstyle{remark}
 \newtheorem{rem}[thm]{Remark}
 \theoremstyle{definition}

 \newcommand{\X}{\mathcal{X}}

 \newcommand{\XX}{\mathbb{X}}
 \newcommand{\PP}{\mathbb{P}}

\def\move-in{\parshape=1.75true in 5true in}
\def\X{\mathbb X} 

 \def\P{\mathbb P} 
 \def\ds{\displaystyle}

\usepackage{xypic}

\begin{document}

\title{{The Secant Line Variety to the Varieties of Reducible Plane Curves}}

\author[M.V.Catalisano]{Maria Virginia Catalisano}
\address[M.V.Catalisano]{Dipartimento di Ingegneria Meccanica, Energetica, Gestionale e dei Trasporti, Universit\`{a} di Genova, Genoa, Italy.}
\email{catalisano@dime.unige.it}

\author[A.V. Geramita]{Anthony V. Geramita}
\address[A.V. Geramita]{Department of Mathematics and Statistics, Queen's University, King\-ston, Ontario, Canada and Dipartimento di Matematica, Universit\`{a} di Genova, Genoa, Italy}
\email{Anthony.Geramita@gmail.com \\ geramita@dima.unige.it  }

\author[A.Gimigliano]{Alessandro Gimigliano}
\address[A.Gimigliano]{Dipartimento di Matematica, Universit\`{a} di Bologna, Bologna, Italy}
\email{alessandr.gimigliano@unibo.it }

\author[Y.S. Shin]{Yong-Su Shin}
\address[Y.S. Shin]{Department of Mathematics, Sungshin Women's University,  Seoul, 136-742, Republic of Korea}
\email{ysshin@sungshin.ac.kr}

\maketitle


\begin{abstract} Let $\lambda =[d_1,\dots,d_r]$ be a partition of $d$.  Consider the variety $\XX_{2,\lambda} \subset \mathbb{P}^N$, $N={d+2 \choose 2}-1$,  parameterizing forms $F\in k[x_0,x_1,x_2]_d$ which are the product of $r\geq 2$ forms $F_1,\dots,F_r$, with $\deg F_i = d_i$. We study the secant line variety $\sigma_2(\XX_{2,\lambda})$, and we determine, for all $r$ and $d$, whether or not such a secant variety is defective.  Defectivity occurs in infinitely many ``unbalanced" cases.  
\end{abstract} 


\section{Introduction}

In 1954 Mammana \cite{M:1} introduced the varieties of {\it reducible plane curves}.  These varieties can be defined as follows:  let $R = k[x_0,x_1, x_2] = \oplus_{i \geq 0} R_i$ ($k = \overline k$ an algebraically closed field) and let $\lambda = [d_1, d_2, \ldots ,d_r]$ be a partition of $d = \sum d_i$ (we write $\lambda \vdash d$ and usually assume $d_1 \geq d_2 \geq \cdots \geq d_r \geq 1$).  

Then, the variety of $\lambda${\it -reducible curves in } $\P^2$, denoted $\X_{2,\lambda}$ is a subvariety of $\P(R_d) = \P^N$ ($N = {d+2\choose 2} - 1$) described by:
$$
\X_{2,\lambda} = \{ [F] \in \P^N \, \mid \, F=F_1\cdots F_r, \deg F_i = d_i \,\}.
$$
(The obvious generalization of these varieties to varieties of reducible hypersurfaces in $\P^n$, for $n>2$, will be denoted by $\X_{n,\lambda}$.  In fact, in recent papers these varieties are often referred to as the varieties of $\lambda${\it-reducible forms}.)

Mammana studied various geometric properties of $\X_{2,\lambda}$ such as its degree, order and singularities.

Not much more was done with these varieties until recently, when they were seen to be extremely useful in the study of vector bundles on surfaces (see \cite{CF:1,P:1}), and in studies related to the classical Noether-Severi-Lefschetz Theorem for general hypersurfaces in $\P^n$ (see \cite{CCG:1, CCG:2}).  In this more modern study of these varieties, secant and join varieties of varieties of $\lambda$-reducible forms have played a key role.  

Indeed, secant and join varieties (including "higher" secant varieties, i.e. varieties of secant $t$ dimensional linear spaces, $t>1$) of most classical varieties have been extensively studied in recent years in part because of their wide-ranging applications in Communications Theory, Complexity Theory and Algebraic Statistics as well as to problems in classical projective geometry and commutative algebra.  This is clearly seen in the following books and papers as well as in their ample bibliographies  (see  \cite{AOP, AH:1, BCS, CGG:3, CGG:6, CS:1,  L:1, R:1, SS:1, V:1}). 

One of the first things considered about secant varieties (both the secant line variety and the higher secant varieties) is their dimension.  The references mentioned above show the great progress made in the last few decades on these questions for Segre, Veronese and Grassmann varieties.  In the case of the varieties of reducible forms there is much less known.

The first significant results about the secant varieties of the varieties of $\lambda$-reducible forms were obtained by Arrondo and Bernardi in \cite{AB} for the case $\lambda = [1, 1, \ldots , 1]$ (where the variety is referred to as the variety of {\it split} or {\it completely reducible} forms).  They found the dimension of these secant varieties for a very restricted, but infinite, class of secant varieties.  This was followed by work of Shin \cite{S:1}  who found the dimensions of the secant line varieties to the variety of split plane curves of any degree.  This was generalized by Abo \cite{A}, again for split curves, to  a determination of the dimensions of all the higher secant varieties.  In the same paper Abo was also able to deal with certain secant varieties for split surfaces in $\P^3$ and split cubic hypersurfaces in $\P^n$.  In all the cases considered, the secant varieties were shown to have the {\it expected dimension}.   Arrondo and Bernardi have speculated if that would always be the case for the varieties of split forms.

In this paper we consider the case of $\lambda$-reducible plane curves of any degree and for every partition $\lambda$.  We find the dimensions of all the secant line varieties in this case.  In sharp contrast to the varieties of split curves,we completely classify those secant line varieties which have the expected dimension and exactly what the defect is for those secant line varieties that do not have the expected dimension.  Roughly speaking we find that the secant line varieties do not have the expected dimension when the partition $\lambda$ is "unbalanced", i.e. if $d_1\geq d_2 + \cdots + d_r$.

\medskip 
The structure of the paper is as follows: in  Section 2 we give a description of how the problem (via Terracini's Lemma) reduces to finding the dimension of the intersection of two generic tangent spaces to our variety, and how this is represented by ideals of zero dimensional schemes. In Section 3 we state our main result (Theorem 3.1) and describe how its proof works; Section 4 is where the induction procedure and the proof of the theorem are given, via a series of Lemmata. In Section 5 we give a more detailed description of the results, also with the aid of some pictures, while in Section 6 we state a few remarks and open questions.


\section{Preliminaries}

 Let $R = k[x_0, ...,x_n]$, where $k$ is an algebraically closed field, and consider the variety $\XX_{n,\lambda} \subset \PP(R_d) = \PP^N\  (N = {d+n\choose n} -1$) of $\lambda$-reducible forms, i.e.
$$
\XX_{n,\lambda}= \{ [F] \in \PP^N \mid F = F_1\cdots F_r, \deg F_i = d_i \}.
$$

\medskip\noindent At a general point $[F] \in \XX_{n,\lambda}$ we have that the $F_i$ are general.  We can thus assume them to be irreducible, and distinct. In addition, we can assume that the hypersurfaces that the $F_i$ define in $\PP^n$ meet transversally.

\medskip Since the map
$$
\PP(R_{d_1}) \times \cdots \times \PP(R_{d_r}) \longrightarrow \XX_{n,\lambda}
$$
is generically finite, we have that
$$
\dim \XX_{n,\lambda} = \sum_{i=1}^r \bigg[{d_i +n\choose n} -1\bigg] = \bigg[ \sum_{i=1}^r{d_i +n\choose n}\bigg] -r.
$$

\medskip Quite generally, if $\XX \subset \PP^m$ is any variety, the (higher) secant varieties to $\XX$, denoted $\sigma_t(\XX)$, are defined as follows:
$$
\sigma_t(\XX) = \overline{ \{ P \in \PP^m \mid P \in \langle Q_1, \ldots , Q_t \rangle , Q_i \in \XX \hbox{ are distinct} \} },
$$
where the overbar denotes Zariski closure.

\medskip Clearly $\sigma_1(\XX) = \XX$, $\sigma_2(\XX)$ is the closure of the set of points on secant lines to $\XX,\dots, \sigma_t(\XX)$ is the closure of the set of points on secant $\PP^{t-1}$'s to $\XX$.

\medskip Since we will mostly be interested in finding dimensions of secant varieties, we recall the following definition:

\medskip the {\it expected dimension} of $\sigma_t(\XX)$ (obtained by a simple count of parameters) is:
$$
\hbox{exp.dim }\sigma_t(\XX): = \min \{ m , \  t\dim \XX + (t-1) \}.
$$

\medskip It is clear that the {\bf actual} dimension of $\sigma_t(\XX)$ can never exceed exp.dim $\sigma_t(\XX)$ although it could be smaller.  In this latter case we say that $\sigma_t(\XX)$ is {\it defective}.  The difference 
$$
\hbox{exp.dim }\sigma_t(\XX) - \dim \sigma_t(\XX)
$$
is called the  {\it $t$-defect} of $\XX$, and denoted by $\delta_t$. So
$\sigma_t(\XX)$ is {\it  defective } if  $\delta_t$ is positive.

One of the most fundamental tools needed to calculate $\dim \sigma_t(\XX)$ is {\it Terracini's Lemma} \cite{T:1}.  Roughly speaking, this lemma says that for $\XX \subset \PP^m$ and $P_1, \ldots , P_t$ general points on $\XX$ the position of the tangent spaces to $\XX$ at these points can determine the dimension of $\sigma_t(\XX)$.  More precisely, if $T_{P_i}(\XX)$ is the (projectivized)  tangent space to $\XX$ at $P_i$ then
$$
\dim \sigma_t(\XX)  = \dim (T_{P_1}(\XX) + \cdots + T_{P_t}(\XX)).
$$

Inasmuch as our main interest is in calculating dimensions of secant varieties to varieties of reducible forms, Terracini's Lemma indicates that we first have to calculate the tangent spaces at general points of those varieties.  This has been discussed in other papers (see \cite{A,CCG:1,S:1}) and we recall those results here.

\medskip If $ [F] = [F_1F_2\cdots F_r]$ is a general point of $\XX_{n,\lambda}$ and $I_ {F} \subset k[x_0, \ldots ,x_n] = R$ is the ideal generated by the $r$ polynomials $F/F_1, \ldots , F/F_r$ then
$$
T_ {F}(\XX_{n,\lambda}) = \PP( (I_ {F})_d ) .
$$
Moreover, $I_{F} = \cap_{1\leq i<j\leq r}(F_i,F_j)$ (see \cite{PS:1}).  

\medskip Note that if $H(R/I_{F},-)$ is the Hilbert function of the graded ring $R/I_{F}$, then 
$$
\dim   \XX_{n,\lambda}  = \dim T_{F}(\XX_{n,\lambda}) = {d+n\choose n} - H(R/I_ {F}, d) -1.
$$
This alternate description of $\dim  \XX_{n,\lambda} $ will be very useful later in this section.

\medskip In the special case we will consider in this paper, namely $n=2$, the ideal $I_{F}$ defines a scheme  consisting of $D = \sum_{1\leq i<j\leq r}d_id_j$ distinct points in $\PP^2$ (which is a union of several inter-related  sets of points which are complete intersections, defined by $(F_i,F_j)$, $i,j\in \{1,\dots,r\}$).

\medskip In case $\lambda =[1,1, \dots ,1] \vdash r$, $n\geq 2$, the codimension 2 subvarieties of $\PP^n$ defined by ideals of the type $I = \cap_{1\leq i<j\leq r}(L_i, L_j)$ have been studied by several authors.  They were first introduced in  \cite{GMS:1} as extremal points sets with maximal Hilbert function and as the support of a family of ${r\choose2}$ fat points.  The name {\it star configuration} was used for such a set of points.  Other applications of such star configurations have figured prominently in the work of \cite{BH,CV:1, CHT,  DSST}.  Generalizations to higher codimension varieties and to not necessarily linear forms have been considered in \cite{AS:1,GHM:1,GHM:2,PS:1}.

Continuing with the case $n=2$, $t=2$, it will be useful to introduce some additional notation.  In this case, the set of $D$ points of $\PP^2$ defined by the ideal $I_{F}$ will be denoted by $Y_F$.  As mentioned before,
$$
T_{F}(\XX_{2,\lambda}) = \PP((I_{F})_d)
$$
and so
$$
\dim \XX_{2,\lambda} = \dim T_{F}(\XX_{2,\lambda}) = \dim(I_{F})_d -1 \eqno(1)
$$

{\it Claim}: $H(R/I_{F},d) = D$.

\medskip This is a well known fact and can be found, for example, in \cite{AS:1, GHM:1}.   In fact, the entire Hilbert function of the ring $R/I_{F}$ is well known (see the same references).  We give a different proof of this here to illustrate how much the ideal $I_{F}$ resembles a complete intersection ideal (when $[F]$ is a general point of $\XX_{2,\lambda}$).

\medskip\medskip

\begin{prop}

Let $[F]$ be a general point of $\XX_{2,\lambda}$
$$\lambda =[d_1,\ldots , d_{r} ], \hskip 2cm  r\geq 2 ,$$
$$d_1\geq d_2 \geq \cdots \geq d_r\geq 1,$$
then

(i) For $j  \geq d-2$,
 $$H(R/I_{F},j) = D \ ;$$
 
(ii) For $j \leq d-1$,
$$H(R/I_{F},j) = {j+2 \choose 2} - \sum _{i=1}^r {\max \{j-d+d_i ; -1\} +2 \choose 2} .$$

\begin{proof}
First we will prove the proposition for $j=d-2$.
Let
$$v = \max \{ n \in \Bbb N \  | \  d_n >1\}.$$
We have $0 \leq v \leq r$ and, with this notation, 
$$\lambda =[d_1,\ldots , d_{v},1,\ldots ,1 ],$$
where the last $r-v$ entries of the partition are $1$.

Now recall that the ideal  $I_{F}$ is generated by the $r$ polynomials $F/F_1, \ldots , F/F_r$,  and
observe that, for $d_i=1$, the degree of  $F/F_i$  is $d-1$. Now we compute the dimension of $(I_{F})_{d-2}$. 

For $v=0$, that is, for $\lambda =[1,\ldots ,1 ],$ there are no forms in $(I_{F})$ of degree ${d-2}$,  so we have 
$$(I_{F})_{d-2} =(0)
.$$
Hence $H(R/I_{F},d-2) =\dim R_{d-2} = {d\choose 2}.$
Since  in this case $j-d+d_i = d-2-d+1 = -1$, and $D=  {d\choose 2}$,  the conclusion follows.

Now let  $v>0$.
 
We have, in degree $d-2$:  
$$(I_{F})_{d-2} = 
\{ M_1\cdot F/F_1+ \cdots + M_v \cdot F/F_v \ | \ M_i \in R_{d_i-2} 
 \} .$$

Hence
in $(I_{F})_{d-2}$ we have at most 
  $\sum _{i=1}^v {d_i \choose 2}
  $
independent forms, that is,
$$ \dim(I_{F})_{d-2} \leq {d_1 \choose 2} + \cdots + {d_v \choose 2}.
$$
On the other hand, obviously we have (recall that $D=  \sum_{1\leq i<j\leq r}d_id_j$):
$$ \dim(I_{F})_{d-2} \geq { d\choose 2} - D $$
$$= 
{d_1+ \cdots +d_r\choose 2}
- \sum_{1\leq i<j\leq r}d_id_j = 
{d_1 \choose 2} + \cdots + {d_r \choose 2}
$$
$$= 
{d_1 \choose 2} + \cdots + {d_v \choose 2}. $$
Hence  
$$\dim(I_{F})_{d-2} ={d_1 \choose 2} + \cdots + {d_v \choose 2} = {d\choose 2} -D. \eqno (2) $$ 
Thus, for $ j  = d-2$, we have
$$H(R/I_{F},j) = D\ .$$

Since  in degree $d-2$ the Hilbert function of the ring $R/I_{F}$ is equal to  the multiplicity $D$ of the scheme $Y_F$, then we also have  $H(R/I_{F},j) = D$,   for $j>d-2$, 
and this completes the proof of  (i).

By noticing that for $j=d-2$,  by (2) we have
$${j+2 \choose 2} - \sum _{i=1}^r {\max \{j-d+d_i ; -1\} +2 \choose 2} =
{d \choose 2} - \sum _{i=1}^v {d_i \choose 2} =D,$$
and for $j=d-1$,  we have
$${j+2 \choose 2} - \sum _{i=1}^r {\max \{j-d+d_i ; -1\} +2 \choose 2} =
{d +1 \choose 2} - \sum _{i=1}^r {d_i+1 \choose 2} $$
$$= \frac {(d_1+ \cdots +d_r+1)(d_1+ \cdots +d_r)} 2-  \sum _{i=1}^r \frac {(d_i+1)d_i}2 
=\sum_{1\leq i<j\leq r}d_id_j = D \ ,
$$
 then in cases $j = d-2$ and $j=d-1$,  we are done also with (ii).

Note that since in degree $d-1$ the dimension of
$I_F$ is exactly the sum of what is generated in degree  $d-1$ by each of the generators of $I_F$, then the same happens for all degrees $j<d-1$, i.e. $H(R/I_{F},j)$ grows as much as possible before degree $d-1$.

By this observation it follows that  the assertion of (ii) also holds  for $j < d-2$.

\end{proof}
\end {prop}

\medskip

In light of the Claim, we can rewrite $(1)$ as
$$
\dim \XX_{2,\lambda} = {d+2\choose 2} - D - 1. 
$$

As we noted earlier (Terracini's Lemma), if $[F]$ and $[G]$ are two general points on $\XX_{2,\lambda}$, then 
$$
\dim \sigma_2(\XX_{2,\lambda}) = \dim ( (I_{F})_d + (I_{G})_d) -1  \eqno{(3)}
$$

By Grassmann's formula
$$
\dim( (I_{F})_d + (I_{G})_d) = \dim (I_{F})_d + \dim  (I_{G})_d - \dim(I_{F} \cap I_{G})_d .
$$

Using $(3)$ above, we obtain
$$
\dim \sigma_2(\XX_{2,\lambda}) = \dim  (I_{F})_d + \dim  (I_{G})_d - \dim(I_{F} \cap I_{G})_d - 1
$$
$$
= 2 \dim\XX_{2,\lambda} + 1 - \dim(I_{F} \cap I_{G})_d .
$$

\medskip Thus, if $\hbox{exp.dim }\sigma_2(\XX_{2,\lambda}) = 2 \dim  \XX_{2,\lambda} +1$  then the 2-defect of $\XX_{2,\lambda}$ is
$$
(2\dim\XX_{2,\lambda} + 1) - (2\dim\XX_{2,\lambda} + 1 - \dim(I_{F} \cap I_{G})_d ) = \dim(I_{F} \cap I_{G})_d .
$$

\medskip If, on the other hand, $\hbox{exp.dim } \sigma_2(\XX_{2,\lambda}) = {d+2\choose 2} - 1$, then the 2-defect of $\XX_{2,\lambda}$ is
$$
\left[ {d+2\choose 2} - 1 \right ] - \bigl[ 2\dim\XX_{2,\lambda} + 1 - \dim(I_{F} \cap I_{G} )_d \bigr] . \eqno{(4)}
$$

But, we noted that $\dim \XX_{2,\lambda} ={d+2\choose 2} - D - 1$, so $(4)$ becomes
$$
 {d+2\choose 2} - 1  -  2\left ({d+2\choose 2} - D - 1\right )- 1 +\dim(I_{F} \cap I_{G} )_d  =
$$
$$
= 2D - {d+2\choose 2} +  \dim(I_{F} \cap I_{G})_d .
$$

So, there is a positive 2-defect in this case if and only if
$$
 \dim(I_{F} \cap I_{G})_d > {d+2\choose 2} - 2D.
$$

\medskip
We summarize the discussion above as follows: 

\medskip
If $\hbox{exp.dim } \sigma_2(\XX_{2,\lambda}) = 2 \dim  \XX_{2,\lambda} +1$, i.e. if $2{d+2\choose 2} - 2D -1 \leq N = {d+2\choose 2} -1$, that is if ${d+2\choose 2} - 2D \leq 0$, then the  2-defect of 
$\XX_{2,\lambda}$ is 
$$\delta_2 := \dim(I_{F}\cap I_{G})_d . $$

If $\hbox{exp.dim } \sigma_2(\XX_{2,\lambda}) = {d+2\choose 2} - 1$, i.e. if ${d+2\choose 2} - 2D > 0$, then $\XX_{2,\lambda}$ has a positive 2-defect if and only if 
$$\dim(I_{F}\cap I_{G})_d > {d+2\choose 2} - 2D , $$
and the 2-defect is:
$$\delta_2 = \dim (I_{F}\cap I_{G})_d -{d+2\choose 2}+2D. $$
 
\bigskip
Now, if we consider $I_{F}\cap I_{G}$, we have that $I_{F} \cap I_{G} = I_{Y_F \cup Y_G}$, where $Y_F$ and $Y_G$ are two sets of $D$ points, both union of related complete intersection: if $[F]=[F_1,\ldots , F_r]$,  then $I_{Y_F} = \cap _{1\leq i<j\leq r} (F_i,F_j)$, and similarly for $[G]$.

We might expect that $Y_F\cup Y_G$ imposes $2D$ independent conditions to curves of degree $d$, thus we have that

$$
\hbox{exp.dim}(I_{Y_F\cup Y_G})_d = \max \left \{{d+2\choose 2} - 2D\ ; 0 \right \} .
$$
Of course the actual dimension of $(I_{Y_F\cup Y_G})_d$ can be bigger.  We write
$$
\delta = \dim (I_{Y_F\cup Y_G})_d - \hbox{exp.dim} (I_{Y_F\cup Y_G})_d  \ , \  (\delta \geq 0).
$$

\medskip
So, when ${d+2 \choose 2} - 2D \leq 0$, we have $\delta_2 = \dim(I_{Y_F\cup Y_G})_d = \delta$, while if ${d+2 \choose 2} -2D >0$ we have  $\delta_2 = \dim (I_{Y_F\cup Y_G})_d - {d+2 \choose 2}+ 2D = \delta .$ 

In conclusion, we get:

\begin{prop} \label{difetto}  Let   
$$
\begin{array}{clllllllllllllllll} 
\delta_2 & =  & \hbox{\rm exp.dim } \sigma_2(\mathbb{X}_{2,\lambda}) - \dim \sigma_2( \mathbb{X}_{2,\lambda}) , \\[.2ex] 
\delta  & =  & \dim(I_{Y_F\cup Y_G})_d-\hbox {\rm exp.dim }(I_{Y_F\cup Y_G})_d.
\end{array}
$$
Then we have: $\delta_2=\delta$. 

%
\end{prop}

\medskip
In the next section we will exploit Proposition \ref{difetto}  in order to prove our main result (see Theorem \ref {theorem}).


\section{The main theorem}

We fix the following notation: 
$$
\begin{array}{llllllllllllllll}
\lambda =[d_1,\ldots , d_r], \quad  r\geq 2 ,\\
d_1\geq d_2 \geq \cdots \geq d_r\geq 1,\\
d= d_1+\cdots + d_r,\\
D=  \sum _{1\leq i<j\leq r}d_id_j .
\end{array}
$$


For any $e= 1,\dots,r$, let
 $$
\begin{array}{lllllllllllllll}
s_e  =  \ds\sum_{ i\neq e } d_i ;\\
p_e  =  D - d_es_e .
\end{array}
$$


Note that 
$$ p_e = 
\left \{
 \begin {matrix}
\ds\sum _{\substack{1\leq i<j\leq r \\  i,j \neq e }}d_id_j,  & \hbox {for } \ r>2, \\
 0,  & \hbox {for } \ r=2.
 \end {matrix}
\right.
$$



If $e=1$, we simplify the notation by writing
$s$ for $s_1$ and $p$ for $p_1$. So
 $$
\begin{array}{lllllllllllll} 
s = d_2+\cdots + d_r,\\[1ex] 
p=  D-d_1 s = \left \{
 \begin {matrix}
\ds\sum _{2\leq i<j\leq r}d_id_j,  & \hbox {for } \ r>2, \\
& \\
 0,  & \hbox {for } \ r=2.
 \end {matrix}
\right. 
\end{array}
$$


\medskip

 If $d_i \geq a_i$, for all $i$, $1 \leq i \leq r$,  $a_i \in \Bbb N$,  we will write $$ [d_1,\ldots , d_r]\succeq [a_1,\ldots , a_r].$$
 
Let $  \mathbb{X}_{2,\lambda}$, $[F]$, $[G]$, $Y_F$, $Y_G$ be as in Section 2.
We set 
$$Z= Y_F\cup Y_G,$$
that is, $Z$ is the   scheme-theoretic union of two sets of $D$ points of $\PP ^2$ 
 defined, respectively, by the ideals $I_ {F}$ and $I_ {G}$ (see Section 2) 
of the two  general points $[F]$ and $[G]$ of
 $\XX_{2,\lambda}$. In this case we simply say that 
 {\it the scheme $Z$ is associated to the partition $\lambda$}.

We will prove the following:

\begin{thm}\label{theorem} 
 $\sigma_2(\XX_{2,\lambda})$ is defective if and only if $d_1\geq s$ and $2p-3s>0$. 
\end{thm}

 {\it Sketch of the proof of Theorem \ref{theorem}.} First note that in  case $r=2$ we never have $2p-3s>0$. This case is dealt with in Lemma \ref{r=2}, where it is proved that, in this case, we always have $\delta_2=0$. 

Given the relationship between the defect of $\sigma_2(\XX_{2,\lambda})$ and the defect of $(I_Z)_d$ established in Proposition \ref{difetto}, the theorem will be proved by working on $\delta$. 

The main point in the theorem is the importance of the condition $d_1\geq d_2+\cdots +d_r$. What happens is that, by working on the examples, it is easy to realize that in most of the cases, when $d_1=d_2+\cdots +d_r$ (so $ d= 2d_1$), the expected dimension of $I_Z$ is $0$, while it is immediate to notice that there is a form of degree $d$ in $I_Z$, hence $\delta \geq 1$. A few computations show that in this case the condition exp.dim $(I_Z)_d=0$ is equivalent to $2p-3s>0$.

For this reason we will consider the partitions $\lambda = [d_1,\dots,d_r] $ with respect to the hyperplane $\mathbb{H} = \{d_1=d_2+\cdots+d_r\}$. We have to prove that if $\lambda$ is ``below $\mathbb{H}$''  (with respect to the $d_1$ direction , i.e. if $d_1<d_2+\cdots +d_r$), then $\delta = 0$, while when we consider $\lambda$ ``above $\mathbb{H}$'' (i.e. with $d_1\geq d_2+\cdots +d_r$) we prove that $\delta =0$ if and only if $2p-3s\leq 0$.

We will use a specialization of $Z$ (described in Remark \ref{specialization}) which allows us (see Lemmata \ref{residuo} and \ref{alpha}) to pass from a point $\lambda$ to another $\lambda '\preceq \lambda$ and relate the dimensions of the spaces $(I_Z)_d$ and $(I_{Z'})_{d'}$ associated to them. In this way, when  $\lambda$ is ``below $\mathbb{H}$'' we can work our way down to a few ``minimal'' cases for which we can directly compute  $\dim (I_Z)_d$ (e.g. using CoCoA \cite{RABCP}), and this way we get that $\delta = 0$ for all  $\lambda$  ``below $\mathbb{H}$'' (see Lemmata \ref{d1<s,I=0}, \ref {(a,a,1)(a+1,a,1)}, \ref {d1<s,I>0}, which we summarize in  Proposition \ref {d1<s}).
  
 Eventually, we will work  ``on $\mathbb{H}$ and above'' (i.e., for $d_1\geq s$), showing,  in that case, that we have $\delta >0$ if and only if $2p-3s > 0$. By using Lemma \ref{residuo} again, we will prove that when $2p-3s\leq 0$ we can ``descend'' to cases with $d_1=s-1$ (below the hyperplane $\mathbb{H}$) and get $\delta =0$ (see Lemma  \ref {d1>=s-1,2p<=3s}), while for $2p-3s > 0$ we can prove that $\delta >0$ by considering forms  in $(I_Z)_d$ which come from the particular structure of $Y_F \cup Y_G$ (see Lemma \ref {d1>=s-1,2p>3s}).

\medskip

Once we find all the defective $\sigma_2(\XX_{2,\lambda})$ we also get a precise description of the defect.

\bigskip                                                                                                                                            
\begin{prop}  \label {cor} When $\sigma_2(\XX_{2,\lambda})$ is defective we have 

$$\delta_2 =  \left \{
\begin {matrix} 
2p-3s            & \hbox { if } {d+2 \choose 2}-2D>0, \\[1ex]
{d_1-s+2 \choose 2} &   \hbox { if }  {d+2 \choose 2} -2D \leq 0 .
\end{matrix}
 \right. $$
\end{prop}

\medskip
\begin {proof}
By Proposition \ref{difetto} and
Lemma \ref {d1>=s-1,2p>3s} (ii) we have that 
$$\delta_2 =\min \left \{  {d_1-s+2 \choose 2}; 2p-3s \right\},$$  
and by a direct computation, we get 
$$ {d_1-s+2 \choose 2} >  2p-3s\ \  \Leftrightarrow \ \  {d+2 \choose 2} - 2D>0.$$
\end {proof}
\section{Proof of the main theorem}

  The next remark describes a specialization we will be using a great deal in what follows.

\medskip

\begin {rem} \label {specialization}

For $\lambda \neq [1,\dots,1] $ and for any $d_e >1$,  ($1 \leq e \leq r$) we construct a specialization $\widetilde Z_e$ of $Z$ that we will use several times in the sequel. If $e$ is clear from the context, we will simply write $\widetilde Z$ instead of  $\widetilde Z_e$.

Recall that the ideals $I_{Y_F}$ and  $I_{Y_G}$ define schemes  made of $D$ distinct points in $\PP^2$, which are  unions of sets of points  defined by $(F_i,F_j)$, and by $(G_i,G_j)$, respectively ($i,j\in \{1,\dots,r\}$, $ i \neq j$).
Consider a curve $\{\widetilde F_e=0\}$ union of a generic line $L= \{ l=0\}$ and a generic curve $\{ F_e'=0\}$ of degree $d_e-1$, so that   $\widetilde F_e= l \cdot F_e'$. Analogously, let $\{\widetilde G_e=0\}$ be union of the same line $L$  and a generic curve $\{ G_e'=0\}$ of degree $d_e-1$, so    $\widetilde G_e= l \cdot G_e'$.  

Let $\widetilde Y_F$ and $\widetilde Y_G$ be the schemes associated to the points $[\widetilde F]$
$= [ F_1\cdots \widetilde F_e\cdots F_r]$ and $[\widetilde G]$ $= [ G_1\cdots \widetilde G_e\cdots G_r]$ of $\Bbb X_{2,\lambda}$, respectively.
Then  $\widetilde Z_e=\widetilde Y_F\cup \widetilde Y_G$ is a specialization of $Z$. 
Now denote by $Y_F'$ and $Y_G'$   the residual schemes of $\widetilde Y_F$ and $\widetilde Y_G$ with respect to $L$, that is,  the schemes defined by the ideals $I_{\widetilde Y_F} :I_L$ and $I_{\widetilde Y_G} :I_L$. Observe that
$$I_{\widetilde Y_F} :I_L= I_{[F_1\cdots F_e'\cdots F_r]} \ \ \ \   { \rm and }\ \ \  \
 I_{\widetilde Y_G} :I_L= I_{[G_1\cdots G_e'\cdots G_r]},$$
 so that each of them consists of 
$D'$ points, where 
$$D' = (d_e-1) s_e +p_e
.$$

Now set  $Z'_e$ ($Z'$ if $e$ is clear from the context) 
$$Z' _e= Y_F' \cup Y_G'.$$
Observe that $[F']$ $=[F_1\cdots F_e'\cdots F_r]$ and $[G']$ $=[G_1\cdots G_e'\cdots G_r]$ are general points in $\XX_{2,\lambda'}$, where $\lambda'_e$ (or simply $\lambda'$),
$$\lambda'_e  = [d_1,\dots, d_e-1,\ldots ,d_r],$$
is a partition of $d'=d-1$, and note that  $Y_F' =  Y_{F'} $  e $Y_G' =  Y_{G'} $.

\end{rem}

 The next two lemmata make use of the specialization we just described in order to ``work our way'' from $Z$ defined by $\lambda$ to a $Z'$ defined by some $\lambda ' \preceq \lambda$.

\begin{lem} 
\label{residuo} 
 Let  $Z'_e$ be as above. Then

 (i) For any $e=1,\dots,r$, if  $1<d_e < s_e$, then
$$
\dim (I_{ Z})_d  \leq \dim (I_{Z'_e})_{d-1}.$$

(ii) Let $e=1$. If $d_1 >1$ and $d_1 \geq s-1$, then
$${\dim } (I_Z)_d \leq \dim (I_{Z'})_{d-1} +d_1-s+1.$$
\begin{proof} 

Let $\widetilde Z$ be the  specialization of  $Z$ described above, so
$\dim (I_{ Z})_d  \leq \dim (I_{\widetilde Z})_d$.

(i)
Since the degree of the scheme $\widetilde Z \cap L$ is $2 \sum _{i \neq e}d_i$ and since $d=d_e+\sum _{i \neq e}d_i<2\sum _{i \neq e}d_i$, by  {\em Bez\'{o}ut}'s Theorem, $L$ is a fixed component for the curves defined by the forms of  $(I_{\widetilde Z})_d$, and the conclusion follows.

(ii) Consider a scheme $ T$ made of $d_1-s+1$ points on the line $L$. Since  the degree of the scheme $(\widetilde Z \cup T) \cap L$ is $2s + d_1-s+1 =  d+1$, then  $L$ is a fixed component for the curves defined by the forms of  $(I_{\widetilde Z\cup T })_d$, and so $\dim (I_{\widetilde Z \cup T} )_d = \dim (I_{Z'})_{d-1} $.
Now
 $$\dim (I_{\widetilde Z \cup T} )_d \geq  {\dim } (I_{\widetilde Z})_d -( d_1-s+1), $$
 hence
 $$\dim  (I_Z)_d \leq \dim (I_{\widetilde Z})_d \leq \dim (I_{\widetilde Z \cup T} )_d + d_1-s+1=  \dim (I_{Z'})_{d-1} +d_1-s+1.
 $$
 
\end{proof}

\end{lem}

\begin{lem} \label{alpha}  
Let $r>2$ and, for $r=3$, let $d_3 \geq 2$.

Let $ \alpha = [a_1,\dots,a_r]$ be a partition of $a$, let $a_1\geq \cdots \geq a_r\geq1$, and $a_1 < a_2 + \cdots + a_r$. Consider the variety  $\mathbb{X}_{2,\alpha}$. Let  
$Z_\alpha $ be  union of the two sets of points of $\PP^2$ defined by the ideals $I_ {F_{\alpha}}$ and $I_ {G_{\alpha}}$ (see Section 2) 
of two  general points ${[F_{\alpha}]}$ and ${[G_{\alpha}]}$ in  $\mathbb{X}_{2,\alpha}$. 

 If $d_1 < s$,  $ [d_1,\dots,d_r] \succeq [a_1,\dots,a_r]$, and ${\dim } (I_{Z_\alpha})_a =0$, then   ${\dim } (I_Z)_d=0$.

\begin{proof}

We prove the lemma by induction on $d$. Obvious for $d=a$, so assume $d>a$.

We want to apply  Lemma \ref{residuo} (i), with  $e = \min\{ n\in \Bbb N\  | \ d_n >a_n\}$, hence we have to check that $1<d_e < s_e$. 
Since $d_e > a_e\geq1$, we get $d_e>1$. The other inequality is obvious for 
 $e=1$. For $e >1$  we prove that $d_e < s_e$  by contradiction.
 
If $d_e \geq s_e$ then  since $d_1 < s$  we have
$$d_e \geq s_e= d_1+s-d_e  >d_1+d_1-d_e,$$
and so $d_e >d_1$, a contradiction. 
Hence we may apply  Lemma \ref{residuo} (i), and we get:
$$
\dim (I_{ Z})_d  \leq \dim (I_{Z'})_{d-1},$$
where $Z'$ is the scheme  associated to the partition
 $$\lambda'  = [d_1,\dots, d_e-1,\ldots ,d_r].$$
If we prove that  $\lambda'$ verifies the hypotheses of the lemma, that is 
\begin {itemize} 
\item [a)] every  number in $\lambda' $ is at least 1;
\item [b)] the largest number in $\lambda' $ is less than the sum of the others;
\end {itemize}
then, by the induction hypothesis, we get $\dim (I_{Z'})_{d-1}=0,$ and we are done.

Since $d_e >a_e\geq 1$, then $d_e-1\geq 1$, and a) is verified.
 As for b), 
  we split the proof into two cases.

Case 1: $d_1 > a_1$. So $e=1$, and
$$\lambda'  =[d_1-1,d_2,\dots,d_r].$$

For $d_1 >d_2$, we have that $d_1-1$ is the largest number in $\lambda'$, and obviously
$ d_1-1 <s$. So, by the induction hypothesis, we get  $\dim (I_{Z'})_{d-1}=0$.

Let $d_1 =d_2$. In this case $d_2$ is the  largest number in $\lambda'$ 
and  we need  to show that 
$$ d_2 <d_1-1+d_3+\cdots+d_r=d_2-1+d_3+\cdots+d_r.$$

Since $d_3+\cdots +d_r >1$,  the  inequality above  holds.
 
Case 2:  $d_1=a_1$. Recall that  $e = \min\{ n\in \Bbb N\  | \ d_e >a_e\}$, and 
 $$\lambda'  = [d_1,\dots, d_e-1,\ldots ,d_r].$$
 The largest number in  $\lambda'$ is still $d_1$, so we have to check that 
 
 $$ d_1 <d_2+ \cdots +(d_e-1)+\cdots+d_r. \eqno  (5) $$
 If  $d_1 \geq d_2+ \cdots +(d_e-1)+\cdots+d_r$, then, since  $d_1=a_1$ but $a<d$, 
  we get
 $$a_2+ \cdots +(a_e-1)+\cdots+a_r < d_2+ \cdots +(d_e-1)+\cdots+d_r \leq d_1 
 $$
 $$= a_1  \leq a_2+ \cdots +a_r-1,$$
 a contradiction. Hence (5) holds and, by the induction hypothesis,  we are done.
 \end{proof}

\end{lem}
\medskip
Lemmata \ref{d1<s,I=0}, \ref {(a,a,1)(a+1,a,1)}, and \ref {d1<s,I>0} consider the case $d_1 < s$.  We summarize those results in   Proposition \ref {d1<s}.

\medskip
\begin{lem} \label{d1<s,I=0}  
If $d_1<s$, and    we are in one of the following cases:

\begin {itemize}
\item   $[d_1,d_2,d_3] \succeq [4,3,3]$;
\item   $[d_1,d_2,d_3] \succeq [6,6,2]$;
\item   $[d_1,d_2,d_3,d_4] \succeq [4,4,1,1]$;
\item   $[d_1,d_2,d_3,d_4] \succeq [2,2,2,1]$;
\item   $[d_1,d_2,d_3,d_4,d_5] \succeq [2,1,1,1,1]$;
\item   $[d_1,\ldots , d_r] \succeq [1,\ldots ,1]$, with  $r \geq6$.

\end {itemize}

 then $(I_{Z})_d = (0).$
 
\begin{proof}
From [26, Corollary 4.3] (for the $[1,\ldots ,1]$ case) and from direct computations using CoCoA (\cite{RABCP}), for the following partitions we get $(I_{Z})_d = (0)$:

\centerline {
$[4,3,3]$, $[6,6,2]$;
$[4,4,1,1]$, $[2,2,2,1]$;
$[2,1,1,1,1]$, and, for
 $r \geq6$, $[1,1,\dots,1]$.}
 
\noindent Now the conclusion  follows immediately from Lemma \ref {alpha} .
\end{proof}
\end{lem}


\begin{lem} \label{(a,a,1)(a+1,a,1)}  
If $r=3$, $a \geq 1$, and $[d_1,d_2,d_3] \in \{[a,a,1],[a+1,a,1]  \}$, 
 then $ (I_{Z})_d$ has the expected positive dimension.

\begin{proof}
By induction on $d$. For $d=3$, that is, for $[d_1,d_2,d_3] =[1,1,1]$ see  Corollary 4.3 in \cite{S:1} or Theorem 4.1 in \cite{A}.
Let $d>3$. 

By an easy computation we get
$$ {\rm exp.dim}   (I_{Z})_d = 
\left\{  \begin {matrix} 
a+3 & \hbox{for } [d_1,d_2,d_3]=[a,a,1], \hfill   \\
a+4 & \hbox{for }  [d_1,d_2,d_3]=[a+1,a,1]  .
\end{matrix}
\right.
$$
By Lemma \ref{residuo} (ii) we have that
$${\dim } (I_Z)_d \leq \dim (I_{Z'})_{d-1} +d_1-s+1,$$
where $Z'$ is associated to the partition
$$ \lambda' = 
\left\{  \begin {matrix} 
[a,a-1,1] & \hbox{for } [d_1,d_2,d_3]=[a,a,1], \hfill   \\
[a,a,1] \hfill & \hbox{for }  [d_1,d_2,d_3]=[a+1,a,1]  ,
\end{matrix}
\right.
$$
and, in both cases, by the induction hypothesis, we get
$$ \dim (I_{Z'})_{d-1} = a+3
.$$

Note that,  in case $[d_1,d_2,d_3]=[a,a,1]$,   $d>3$ implies $a \geq 2$. Hence $a-1 \geq 1$. 
Thus, for $[d_1,d_2,d_3]=[a,a,1]$ we have
$$ a+3 = {\rm exp.dim}   (I_{Z})_d \leq {\dim } (I_Z)_d \leq a+3 +d_1-s+1=a+3;$$
for $[d_1,d_2,d_3]=[a+1,a,1]$ we have
$$ a+4 = {\rm exp.dim}   (I_{Z})_d \leq {\dim } (I_Z)_d \leq a+3 +d_1-s+1=a+4;$$
and we are done.
\end{proof}
\end{lem}


\begin{lem} \label{d1<s,I>0}  Let $d_1 <s$.
If we are in one of the following cases:

\begin {itemize}
\item   $r =3$, $[d_1,d_2,d_3]\in \{ [a,a,1]$, $ [2,2,2]$, $[3,2,2]$, $[3,3,2]$,  $[4,3,2]$, $[4,4,2],$  $[5,4,2],$
 $[5,5,2], $ $ [6,5,2],[3,3,3] \}$;
\item   $r =4$, $[d_1, d_2, d_3, d_4]\in \{ [1,1,1,1], [2,1,1,1],   [2,2,1,1],  [3,2,1,1]$, $ [3,3,1,1]$, $[4,3,1,1] \}$;
\item   $r =5$, $[d_1,\dots,d_5]=[1,1,1,1,1]$;

\end {itemize}

 then $ (I_{Z})_d$ has the expected positive dimension.

\begin{proof}

For  $d_1=1$, that is if $\lambda= [1,\dots,1]$ see  Corollary 4.3 in \cite{S:1} or Theorem 4.1 in \cite{A}, so let $d_1 >1$.

For $ [a,a,1]$ see Lemma \ref {(a,a,1)(a+1,a,1)}. 

Direct computations using CoCoA (see \cite{RABCP}) (or ad hoc specializations) cover the cases:
$[2,2,2]$, $[3,3,2]$, $[4,4,2]$, $[5,5,2]$, $[3,3,3]$, $[2,2,1,1]$, and $[3,3,1,1]$.

All the other cases follow from the previous ones by applying Lemma \ref {residuo} (ii).
For instance, consider $ [3,2,1,1]$. By  Lemma \ref {residuo} (ii) we have
$${\dim } (I_Z)_d \leq \dim (I_{Z'})_{d-1} +d_1-s+1,$$
and $Z'$ is associated to $ [2,2,1,1]$, so 
$$\dim (I_{Z'})_{d-1}= 28- 26 =2.$$
Since  ${\rm exp.dim } (I_Z)_d = 36-34=2$, we have
$$ 2 \leq {\dim } (I_Z)_d \leq \dim (I_{Z'})_{d-1} +d_1-s+1 = 2+3-4+1=2,$$
and the conclusion follows.

For the other cases the following table is useful. Note that for these cases $d_1=s-1$.
$$
\begin{array}{|c|c|c|c|c|ccccccccc}
\hline
 {\rm case }      &   {\rm exp.dim}  (I_Z)_d   &  \lambda'     & \dim  (I_{Z'})_{d-1}    \\ 
\hline
{[3,2,2]}              &   4                                    &   [2,2,2]       &    4          \\
{[4,3,2]}              &  3                                    &  [3,3,2]      &    3           \\
{[5,4,2]}            &   2                                   &    [4,4,2]    &    2                \\
{[6,5,2]}              &  1                                    &   [5,5,2]     &    1         \\
{[2,1,1,1]}           &  3                                    &   [1,1,1,1]  &    3            \\
{[3,2,1,1]}           &   2                                   &   [2,2,1,1]  &    2           \\
{[4,3,1,1]}           &  1                                    &  [3,3,1,1]   &    1         \\
\hline              
\end{array}
$$


%
\end{proof}
\end{lem}

\medskip
 The following proposition gives a complete picture of what happens ``below $\mathbb{H}$''.

\begin{prop} \label {d1<s}  Let  $d_1 <s$. Then 

(i) $ (I_{Z})_d $ has the expected dimension;

(ii) $\dim (I_{Z})_d = (0)$
if and only if we are in one of the following cases:

\begin {itemize}
\item   $[d_1,d_2,d_3] \succeq [4,3,3]$;
\item   $[d_1,d_2,d_3] \succeq [6,6,2]$;
\item   $[d_1,d_2,d_3,d_4] \succeq [4,4,1,1]$;
\item   $[d_1,d_2,d_3,d_4] \succeq [2,2,2,1]$;
\item   $[d_1,d_2,d_3,d_4,d_5] \succeq [2,1,1,1,1]$;
\item   $[d_1,\ldots , d_r] \succeq [1,\ldots ,1]$, with  $r \geq6$.

\end {itemize}
\begin{proof}
This follows from Lemmata \ref {d1<s,I=0} and \ref {d1<s,I>0}, by noticing that the cases listed in Lemma \ref {d1<s,I>0} are the ones  left out by Lemma \ref {d1<s,I=0}.
\end{proof}
\end{prop}


Lemma \ref {r=2} below, deals with the special case in which $r=2$.  This was not dealt with in the preceding lemmata since when $r=2$ we never have $d_1< s \ (=d_2)$.  It's not convenient to even consider this case in the successive lemmata since in order to study the case $d_1\geq s$ we actually have to begin with the case in which $d_1 = s-1$.

\begin{lem} \label {r=2}  For $r=2$, $ (I_{Z})_d $ has the expected positive dimension.
\begin {proof}

In this case the dimension of $(I_{Z})_d$ is always positive, in fact
$${\rm exp.dim } (I_{Z})_d = {d_1+d_2+2 \choose 2} - 2d_1d_2= 
\frac{ (d_1-d_2)^2 +3(d_1+d_2)+2 }{2}.
$$
For $d=2$, that is, for $[d_1,d_2]= [1,1]$, we trivially have that the dimension of $(I_{Z})_d$ is as expected.
For $d >2$, by induction on $d$, 
 by Lemma \ref {residuo} (ii), and easy computations we get:
$${\dim } (I_Z)_d \leq \dim (I_{Z'})_{d-1} +d_1-s+1= {\rm exp.dim } (I_{Z})_d.$$
It follows that for $r=2$, $ (I_{Z})_d $ has the expected dimension.
\end{proof}
\end{lem}
 
\medskip
  The Lemmata \ref {d1>=s-1,2p<=3s}, \ref {d1>=s-1,2p>3s} deal with the case $d_1 \geq s$  and complete the proof of Theorem \ref{theorem}.

\medskip
 \begin{lem} \label{d1>=s-1,2p<=3s}  
Let $r>2$. If $d_1\geq s-1$, and $2p -3s\leq  0$, then  $(I_{Z})_d$ has the expected dimension.

\begin{proof} By induction on $d_1$.
If $d_1 = s-1$, the conclusion follows from Proposition \ref{d1<s}. 

Let $d_1 \geq s$.
Observe that if $d_1=d_2$, we have $s-d_1 \geq d_2+d_3 -d_1  >0 $, a contradiction. So  $d_1>d_2$.

 The expected dimension of $(I_Z)_d$ is

$$ {\rm exp.dim } (I_Z)_d = {d+2 \choose 2} -2d_1s -2p = {d_1-s+2 \choose 2} +3s -2p .
$$

Now let $e=1$ and consider  the scheme  $Z'$ associated to the partition $[d_1-1,d_2,\dots,d_r].$
Since $d_1-1 \geq s-1$, by the induction hypothesis and by Lemma \ref {residuo} (ii) we have 
$${\dim } (I_Z)_d \leq \dim (I_{Z'})_{d-1} +d_1-s+1
={d+1 \choose 2} -2(d_1-1) s -2p +d_1-s+1$$
$$=  {d_1-s+2 \choose 2} +3s -2p ={\rm exp.dim } (I_Z)_d ,$$
and the conclusion follows.
\end{proof}
\end{lem}

 \begin{lem} \label{d1>=s-1,2p>3s}  
Let $r>2$, and $2p - 3s >0$, then

(i) for  $d_1\geq s-1$, 
$$ {\dim } (I_Z)_d = {d_1-s+2 \choose 2} ;$$

(ii)  if $d_1\geq s$,  then  $(I_{Z})_d$  is defective with defect 
$$\delta =\min \left \{  {d_1-s+2 \choose 2}; 2p-3s \right\}.$$

\begin{proof} 
(i) By induction on $d_1$.
If $d_1 = s-1$, then, by Lemma \ref{d1<s}, we have that $ {\dim } (I_Z)_d = 
{\rm exp.dim } (I_Z)_d  = \max \left \{ 0; {d_1-s+2 \choose 2}+3s-2p\right\}=0={d_1-s+2 \choose 2} .$

Let $d_1 \geq s$.

Now recall that $Z = Y_F \cup Y_G$ is  union of  two sets of points of $\PP^2$ defined by the ideals of two  general points $[F]$ and $[G]$  of  $\mathbb{X}_{2,\lambda}$, 
$$ F =  F_1\cdots F_r ; \ \ \ \  G = G_1\cdots G_r ,
$$
and $\deg F_i =\deg G_i = d_i.$ Hence
$$  F_2\cdots F_r \cdot G_2\cdots G_r ,
$$
is a form of degree $2s$ in the ideal $I_Z$. It follows that there are at least 
${d-2s+2\choose 2}= {d_1-s+2\choose 2}$ independent  forms in $(I_Z)_d,$ that is,
$$ {\dim } (I_Z)_d \geq {d_1-s+2 \choose 2} .$$
By Lemma \ref {residuo} and by the induction hypothesis we get
$$ {\dim } (I_Z)_d \leq {\dim } (I_{Z'})_{d-1} +d_1-s+1 
={d_1-s+1 \choose 2} +d_1 -s +1 ={d_1-s+2 \choose 2},$$
where $Z'$ is the scheme associated to
the partition $[d_1-1; d_2; \ldots ; d_r]$ of $d-1$, and the conclusion follows.

 (ii) The expected dimension of  $ (I_Z)_d $ is
$$ {\rm exp.dim } (I_Z)_d =\max \left \{ 0; {d_1+s+2 \choose 2}-2d_1s-2p \right\}$$
$$=\max \left \{ 0; {d_1-s+2 \choose 2}+3s-2p\right\}.$$
Since, by (i),
$\dim (I_{Z})_d = {d_1-s+2 \choose 2} >0,$ then
$(I_{Z})_d$  is defective with defect 
$$\delta = {d_1-s+2 \choose 2} -\max \left \{ 0; {d_1-s+2 \choose 2}+3s-2p\right\}
$$
$$=\min \left \{  {d_1-s+2 \choose 2}; 2p-3s \right\}.$$  
\end{proof}
\end{lem}

\section { A pictorial description of the Main Theorem}


 As we noticed, Theorem \ref{theorem} directs us to the hyperplane $\Bbb H$ in $\Bbb N^r$ defined by $x_1=x_2+ \cdots +x_r$. More precisely, from Theorem \ref{theorem} we see that if $(d_1, \dots, d_r) \in \Bbb N^r$ is such that $d_1< d_2 + \cdots+ d_r$,  i.e.  is ``below'' $\Bbb H$, then $\sigma_2(\XX_{2,\lambda})$, $\lambda = [d_1, \dots, d_r]$, is not defective.

Thus, it only remains to give a more detailed description of defective and non-defective cases ``above'' $\Bbb H$. From Theorem \ref{theorem} we know that $\delta > 0$ if and only if $2p-3s > 0$.  The next proposition  gives a complete list of the defective cases.


\begin{prop}\label{2p>3s}   $i)$\ We have $2p-3s >0$ if and only if we are in one of the following cases:  

\begin {itemize}
\item   $r =3$, $d_3=2$ and $d_2 \geq 7$;
\item   $r =3$, $d_3=3$ and $d_2 \geq 4$;
\item   $r =3$, $d_3 \geq 4$;
\item   $r =4$, $d_3 \geq 2$;
\item   $r =4$, $d_2 \geq 5$;
\item   $r =5$, $d_2 \geq 2$;
\item   $r \geq 6$.
\end {itemize}

$ii)$\ If $d_1\geq s$ and $2p-3s >0$,  then $\delta >0$ and $\sigma_2(\mathbb{X}_{2,\lambda})$ is defective if and only if we are in one of the cases listed above.
%
%
%
\end{prop}

\begin {proof}  First note that $ii)$ is immediate from Lemma 4.10. 

$i)$\ Obvious for $r=2$,  since in this case $p=0$.

If $r=3$,  the cases when $d_3\leq 3$ follow because we have: 

$$
2p-3s= \left\{  \begin {matrix} 
\hfill -d_2-3 & \hbox{for } d_3=1,   \\
\hfill d_2-6 & \hbox{for } d_3=2,  \\
\hfill 3d_2-9 & \hbox{for } d_3=3 .
\end{matrix}
\right. 
$$

When $d_3 \geq 4$, since $d_2 \geq d_3$, we get $2p-3s\geq 5d_2-12>0$. This finishes the case $r=3$.
 
 If $r=4$, from the equality
$$2p-3s=2d_2(d_3-1)+2d_3(d_4-1)+2d_4(d_2-1)-(d_2+d_3+d_4),$$ 
it is not hard to check that if  $d_3=d_4=1$ and $d_2\leq 4$ then $2p-3s\leq 0$, while it is positive otherwise (i.e. in the cases given in the statement of the proposition).

 If $r=5$, the conclusion follows from the equality
$$2p-3s=(2d_5+d_3)(d_2-1)+(2d_4+d_2)(d_3-1)+(2d_2+d_5)(d_4-1)+(2d_3+d_4)(d_5-1).$$ 
Which shows that we always have $2p-3s\geq0$, and $2p-3s=0$ only for $[d_1,1,1,1,1]$, i.e.  for $d_2=1$.

 Now let $r\geq6$. We will work by induction on $r$. We use as our starting point the case $r=5$, since we just noticed that $2p-3s\geq 0$ in that case.
Now let $r>5$ and consider the equality
$$2p-3s\ =\ 2\sum _{2\leq i<j\leq r-1}d_id_j - 3(d_2+\cdots +d_{r-1})  +2(d_2+\cdots+d_{r-1})d_r-3d_r. 
$$
Since $2(d_2+\cdots+d_{r-1})d_r-3d_r > 0$, and, by  the induction hypothesis, $$2\sum _{2\leq i<j\leq r-1}d_id_j - 3(d_2+\cdots d_{r-1}) \geq 0,$$
   then,  for $r\geq6$, $2p-3s$   is positive  and     we are done. 
\end {proof}

\medskip

\begin{rem}\label{2p<=3s}  From Proposition \ref{2p>3s} we have $2p-3s \leq 0$ if and only if we are in the following cases.

\begin {itemize}
\item   $r =2$;
\item   $r =3$, $[d_2,d_3]\in \{ [a,1], [2,2],[3,2],[4,2],[5,2],[6,2],[3,3] \}$, where $a\geq 1$;
\item   $r =4$, $[d_2, d_3, d_4]\in \{ [1,1,1][2,1,1],[3,1,1], [4,1,1]  \}$;
\item   $r =5$, $[d_2,\dots,d_5]=[1,1,1,1]$.
\end {itemize}

Note that for $d_1 \geq s$ in the cases above, we have $\delta = 0$ and $\sigma_2(\XX_{2,\lambda})$ is not defective.
\end{rem}


From Remark \ref {2p<=3s}  we see that to pictorially describe all the non-defective $\lambda$ above $\Bbb H$, there are exactly five cases to consider for $r$. Recall that we are only considering $$\lambda = [d_1, \ldots, d_r] \ \ \hbox{with }  \ \ d_1 \geq d_2 + \cdots + d_r.
$$

 
  
  \begin {itemize}
  
\item $r =2$: in this case, for every $\lambda$, we get that  $\sigma_2(\XX_{2,\lambda})$ is non-defective.

   
\item $r =3$:   $\sigma_2(\XX_{2,\lambda})$ is defective except for

 $$
 \begin{array}{llllllllllllllllll}
 \lambda
 & = & [d_1,d_2,1] , [d_1,2,2], [d_1,3,2] , [d_1,4,2]    , [d_1,5,2] , [d_1,6,2]  , [d_1,3,3] \\
 &    & \hbox {with } \ \ d_1 \geq d_2+1, 4,5,6,7,8 \hbox { and } 6, \  \hbox {respectively  (See Figure~\ref{FIG:r=3}).}
 \end{array}
$$

 


 
{
\begin{figure}[ht] 
\centering

\definecolor{uuuuuu}{rgb}{0.26666666666666666,0.26666666666666666,0.26666666666666666}
\begin{tikzpicture}[line cap=round,line join=round,>=triangle 45,x=0.3cm,y=0.3cm]
\clip(-1.1,-1.2) rectangle (15.005358755092532,14.987074591164674);
\draw[smooth,samples=100,domain=1.6:15.0] plot(\x,{3.0*(\x)/(2.0*(\x)-3.0)});
\draw [->] (-0.54,0.0) -- (15.0,0.0);
\draw [->] (0.0,-0.48) -- (0.0,15.0);
\draw [dotted] (-0.46,-0.42)-- (7.0,7.0);

\draw (5.3,8.2) node[anchor=north west] {\scriptsize$d_2=d_3$};

\draw (4,4.3) node[anchor=north west] {\scriptsize$2d_2d_3-3(d_2+d_3)=0$};

\draw (4,11.723131742684735) node[anchor=north west] {\parbox{5 cm}{\scriptsize$d_2,d_3\in \mathbb N^+\\  2d_2d_3-3(d_2+d_3)\le 0  \\ d_2\ge d_3$}};
\draw [dotted] (0.0,1.0)-- (2.54,0.98);
\draw [dotted] (0.0,2.0)-- (2.56,1.98);
\draw [dotted] (0.0,3.0)-- (3.58,2.98);
\draw [dotted] (0.0,4.0)-- (2.48,3.98);
\draw [dotted] (0.0,5.0)-- (2.52,5.0);
\draw [dotted] (0.0,6.0)-- (2.44,6.0);
\draw [dotted] (0.0,7.0)-- (2.46,7.0);
\draw [dotted] (0.0,8.0)-- (2.48,8.0);
\draw [dotted] (0.0,9.0)-- (2.5,9.0);
\draw [dotted] (0.0,10.0)-- (2.46,10.0);
\draw [dotted] (0.0,11.0)-- (2.48,11.0);
\draw [dotted] (0.0,12.0)-- (2.5,12.0);
\draw [dotted] (0.0,13.0)-- (2.48,13.0);
\draw [dotted] (0.0,14.0)-- (2.46,14.0);
\draw (0.43,15) node[anchor=north west] {$\vdots$};

\draw [dotted] (1.0,15.0)-- (1.0,0.0);
\draw [dotted] (2.0,15.0)-- (2.0,0.0);

\draw (0.4,0.1) node[anchor=north west] {\scriptsize$1$};
\draw (1.4,0.1) node[anchor=north west] {\scriptsize$2$};
\draw (2.4,0.1) node[anchor=north west] {\scriptsize$3$};
\draw (3.4,0.1) node[anchor=north west] {\scriptsize$4$};
\draw (4.4,0.1) node[anchor=north west] {\scriptsize$5$};
\draw (5.4,0.1) node[anchor=north west] {\scriptsize$6$};
\draw (6.4,0.1) node[anchor=north west] {\scriptsize$7$};
\draw (7.4,0.1) node[anchor=north west] {\scriptsize$8$};
\draw (8.3,0.1) node[anchor=north west] {\scriptsize$10$};
\draw (9.4,0.1) node[anchor=north west] {\scriptsize$11$};
\draw (10.5,0.1) node[anchor=north west] {\scriptsize$12$};


\draw (12,0.15) node[anchor=north west] {\scriptsize$d_3$};

\draw (-1.1,1.7) node[anchor=north west] {\scriptsize$1$};
\draw (-1.1,2.7) node[anchor=north west] {\scriptsize$2$};
\draw (-1.1,3.7) node[anchor=north west] {\scriptsize$3$};
\draw (-1.1,4.7) node[anchor=north west] {\scriptsize$4$};
\draw (-1.1,5.7) node[anchor=north west] {\scriptsize$5$};
\draw (-1.1,6.7) node[anchor=north west] {\scriptsize$6$};
\draw (-1.1,7.7) node[anchor=north west] {\scriptsize$7$};
\draw (-1.1,8.7) node[anchor=north west] {\scriptsize$8$};
\draw (-1.1,9.7) node[anchor=north west] {\scriptsize$9$};
\draw (-1.4,10.7) node[anchor=north west] {\scriptsize$10$};
\draw (-1.4,11.7) node[anchor=north west] {\scriptsize$11$};
\draw (-1.4,12.7) node[anchor=north west] {\scriptsize$12$};

\draw (-1.43,14) node[anchor=north west] {\scriptsize$d_2$};


\draw [->] (9,2.9) -- (7.416336375266204,1.9474260777500685);

\begin{scriptsize}
\draw [fill=black] (0.9839144415941035,1.0161722730064675) circle (1.5pt);
\draw [fill=black] (1.9865953679950863,2.013476894172056) circle (1.5pt);
\draw [fill=uuuuuu] (3.0,3.0) circle (1.5pt);
\draw [fill=black] (1.0,2.0) circle (1.5pt);
\draw [fill=black] (1.0,3.0) circle (1.5pt);
\draw [fill=black] (1.0,4.0) circle (1.5pt);
\draw [fill=black] (1.0,5.0) circle (1.5pt);
\draw [fill=black] (1.0,6.0) circle (1.5pt);
\draw [fill=black] (1.0,7.0) circle (1.5pt);
\draw [fill=black] (1.0,8.0) circle (1.5pt);
\draw [fill=black] (1.0,9.0) circle (1.5pt);
\draw [fill=black] (1.0,10.0) circle (1.5pt);
\draw [fill=black] (1.0,11.0) circle (1.5pt);
\draw [fill=black] (1.0,12.0) circle (1.5pt);


\draw [fill=black] (2.0,3.0) circle (1.5pt);
\draw [fill=black] (2.0,4.0) circle (1.5pt);
\draw [fill=black] (2.0,5.0) circle (1.5pt);
\draw [fill=uuuuuu] (2.0,6.0) circle (1.5pt);
\end{scriptsize}
\end{tikzpicture}
\vskip -1pc

\caption{} \label{FIG:r=3}

\end{figure} }

  Here the $\bullet$ (under the hyperbola) represent non-defective cases, while all other points 
above the line $d_2 = d_3$ represent defective cases (the locus under the line we
are not interested in because we are assuming, w.l.o.g., that $d_2\geq d_3$). 
 
\item  $r =4$:   $\sigma_2(\XX_{2,\lambda})$ is defective except for  
 $$\begin{array}{llllllllllllllllllll}
 \lambda
 & = & [d_1,1,1,1] , [d_1,2,1,1], [d_1,3,1,1] , [d_1,4,1,1], \\
 &    & \hbox {with } \ \ d_1 \geq 3,4,5 \hbox { and }  6,\  \hbox {respectively  (See Figure~\ref{FIG:r=4}). }
 \end{array} 
 $$



\begin{figure}[ht]
\centering
\vskip-1.2pc
\definecolor{zzttqq}{rgb}{0.6,0.2,0.0}
\definecolor{uuuuuu}{rgb}{0.26666666666666666,0.26666666666666666,0.26666666666666666}
\definecolor{qqqqff}{rgb}{0.0,0.0,1.0}
\begin{tikzpicture}[line cap=round,line join=round,>=triangle 45,x=0.6cm,y=0.6cm]
\clip(-2.3040120958245542,-0.8636160472153103) rectangle (6.035382058226799,12);

\draw [->] (-1.964490795368808,-0.6555223469359805) -- (6.0,2.0);
\draw [->] (1.4776606378231216,-0.7415761327657788) -- (-2.0,1.0);

\draw (.91,0.47) node[anchor=north west] {\tiny$1$};
\draw (2.1,   0.87) node[anchor=north west] {\tiny$2$};
\draw (3.33,1.26) node[anchor=north west] {\tiny$3$};
\draw (4.4,1.65) node[anchor=north west] {\tiny$4$};

\draw (-1.12,0.53) node[anchor=north west] {\tiny$1$};

\draw [dotted] (1.6101218369259607,1.203373945641987)-- (1.5809251808188796,5.2133458466562566);
\draw [dotted] (2.8,1.6)-- (2.8028889396020147,6.745103234426664);
\draw [dotted] (3.986484106267421,1.9954947020891403)-- (4.0,10.0);
\draw [dotted] (0.0,6.0)-- (4.78,7.8);
\draw [dotted] (3.632,1.184)-- (3.64,7.36);
\draw [dotted] (2.82,7.74)-- (3.64,7.3);
\draw [dotted] (4.02,8.22)-- (4.8,7.8);
\draw [dotted] (4.02,8.22)-- (2.8,7.74);
\draw [dotted] (2.4022683408341097,0.8045048824976484)-- (2.3999941097275714,5.885774073076522);
\draw [dotted] (3.64,6.36)-- (2.8,6.74);
\draw [dotted] (2.82,6.76)-- (1.64,6.26);
\draw [dotted] (2.42,5.86)-- (1.58,6.22);
\draw [dotted] (1.216,0.392)-- (1.1752763090478346,4.438380308636833);
\draw [dotted] (0.0,4.0)-- (2.3812916588560515,4.8847073437210184);
\draw [dotted] (0.44,0.8)-- (0.4436267138365634,4.804205106242469);
\draw [dotted] (2.3840886837447175,4.883732236156737)-- (1.6047228105848852,5.217746181796666);
\draw [dotted] (0.4436267138365634,4.804205106242469)-- (1.1752763090478346,4.422474882653979);
\draw [dotted] (0.0,3.0)-- (1.1709057225353288,3.420289791135454);
\draw [dotted] (-0.76,3.38)-- (0.48,3.8);
\draw [dotted] (1.6047228105848852,5.249557033762373)-- (0.4536399824934189,4.809992162466653);
\draw [dotted] (-0.78,3.38)-- (-0.782,0.40599999999999997);
\draw [dotted] (0.0,5.0)-- (3.64,6.36);
\draw [dotted] (4.78,7.8)-- (4.744049798912714,1.616);
\draw [dotted] (0.44269066046719185,3.7707213933657893)-- (1.1853707481177598,3.435378406690916);
\draw [dotted] (-0.76,3.38)-- (0.0,3.0);
\draw [dotted] (3.9864917521383028,2.0000228091600687)-- (4.744049798912714,1.6160000000000005);
\draw [dotted] (-0.782,0.40600000000000014)-- (3.9864917521383028,2.0000228091600687);
\draw [dotted] (1.610105347312631,1.2056386873942337)-- (2.402247932138703,0.8501036234919165);
\draw [dotted] (2.8000019091612747,1.6034001513366527)-- (3.632,1.184);
\draw [dotted] (0.4400131307372575,0.8144974672787895)-- (1.2157495009564916,0.41689004246141903);
\draw [dotted] (0.0,1.0)-- (0.7973298162635007,-0.37524141658631427);
\draw [dotted] (0.0,1.0)-- (-1.1900940051026245,-0.3901844528371874);

\draw (0.44,3.7707213933657893)-- (0.44,9.997936338793036);
\draw (1.58,5.2133458466562566)-- (1.58,9.980725581627077);
\draw (2.78,6.74)-- (2.78,10.015147095958996);
\draw (3.98,8.16397819266512)-- (3.98,9.999999999999996);
\draw (-0.5,1.31) node[anchor=north west] {\tiny$1$};
\draw (-0.5,2.31) node[anchor=north west] {\tiny$2$};
\draw (-0.5,3.31) node[anchor=north west] {\tiny$3$};
\draw (-0.5,4.31) node[anchor=north west] {\tiny$4$};
\draw (-0.5,5.31) node[anchor=north west] {\tiny$5$};
\draw (-0.5,6.31) node[anchor=north west] {\tiny$6$};
\draw (-0.5,7.31) node[anchor=north west] {\tiny$7$};
\draw (-0.5,8.25) node[anchor=north west] {\tiny$8$};
\draw (-0.5,9.31) node[anchor=north west] {\tiny$9$};

\draw (0.25,-0.33) node[anchor=north west] {\tiny$-1$};
\draw (-1.7,-0.28) node[anchor=north west] {\tiny$-1$};

\draw (-0.18,3.85) node[anchor=north west] {\tiny$(3,1,1,1)$};

\draw (-0.18,4.87) node[anchor=north west] {\tiny$(4,1,1,1)$};

\draw (0.76,5.29) node[anchor=north west] {\tiny$(4,2,1,1)$};
\draw (1.99,6.8) node[anchor=north west] {\tiny$(5,3,1,1)$};

\draw (3.1,8.2) node[anchor=north west] {\tiny$(6,4,1,1)$};

\draw [->] (0.0,-0.4834147752763842) -- (0.0,11.0);

\fill[color=zzttqq,fill=zzttqq,fill opacity=0.1] (-0.0,1.0) -- (5.143551914172531,6.934421563252218) -- (5.143551914172531,9.997936338793036) -- (-1.2932712658963779,7.743327150052323) -- (-1.2932712658963779,3.4922701300602905) -- cycle;

\draw [color=black,line width=1pt] (-0.0,1.0)-- (5.143551914172531,6.934421563252218);
\draw [color=black,line width=1pt] (5.143551914172531,6.934421563252218)-- (5.143551914172531,9.997936338793036);
\draw [color=black,line width=1pt] (5.143551914172531,9.997936338793036)-- (-1.2932712658963779,7.743327150052323);
\draw [color=black,line width=1pt] (-1.2932712658963779,7.743327150052323)-- (-1.2932712658963779,3.4922701300602905);
\draw [color=black,line width=1pt] (-1.2932712658963779,3.4922701300602905)-- (-0.0,1.0);

\draw [shift={(4.024852698385154,8.535021979686464)},line width=1pt] plot[domain=-1.62075472251684:1.5208379310729534,variable=\t]({1.0*0.3446451436658829*cos(\t r)+-0.0*0.3446451436658829*sin(\t r)},{0.0*0.3446451436658829*cos(\t r)+1.0*0.3446451436658829*sin(\t r)});

\draw [line width=1pt](3.996925134175838,8.178975044725279)-- (2.475884553448785,8.035910021873637);
\draw  [line width=1pt](4.042063455551114,8.879237123005655)-- (2.940574996929696,9.034133937499297);

\draw (-0.6,10.3) node[anchor=north west] {\tiny$d_1$};

\draw (4.9,1.87) node[anchor=north west] {\tiny$d_2$};

\draw (-1.79,0.8) node[anchor=north west] {\tiny$d_3$};

\draw (-0.5,0.05) node[anchor=north west] {\tiny$0$};

\draw (-1.9,11.7) node[anchor=north west] {\tiny$\mathbb H: d_1=d_2+d_3+d_4 \text{ in $\mathbb N^4$ with } d_4=1$};

\draw (2.8,10.6) node[anchor=north west] {\tiny$2p-3s=0$};

\draw [->] (3.735399055139465,10.119976253242557) -- (3.5,9);

\draw [->] (1.482354480686567,11) -- (1.05,8.6);

\begin{scriptsize}
\draw [fill=uuuuuu] (2.78,6.738694726568546) circle (2pt);
\draw [color=black] (2.78,7.74) circle (2pt);
\draw [color=black] (2.78,8.775590373995326) circle (2pt);
\draw [color=black] (2.78,9.757009666570312) circle (2pt);
\draw [fill=black] (3.98,8.16397819266512) circle (2pt);
\draw [fill=uuuuuu] (0.44,3.7707213933657893) circle (2pt);
\draw [color=black] (0.44,4.804205106242469) circle (2pt);
\draw [color=black] (1.58,6.218049027950298) circle (2pt);
\draw [color=black] (0.44,6.831175116810089) circle (2pt);
\draw [color=black] (0.44,5.798535587902606) circle (2pt);
\draw [color=black] (0.44,7.846645129127278) circle (2pt);
\draw [color=black] (0.44,8.827611539364558) circle (2pt);
\draw [color=black] (1.58,7.261461389254693) circle (2pt);
\draw [color=black] (1.58,8.294155492929267) circle (2pt);
\draw [color=black] (1.58,9.292366398990016) circle (2pt);
\draw [fill=black] (1.58,5.2133458466562566) circle (2pt);
\draw [color=black] (3.98,9.257917861117466) circle (2pt);
\end{scriptsize}
\end{tikzpicture}
\vskip -.5pc
\caption{} \label{FIG:r=4}
\end{figure}

\vfill\eject
\item $r =5$:   $\sigma_2(\XX_{2,\lambda})$ is defective except for  
 $$
 \lambda= [d_1,1,1,1,1] , \ \ \ \ d_1 \geq 4   \text{ (See Figure~\ref{FIG:r=5}).}
 $$ 
 
\medskip 
 \medskip 
%
%
%
%
\begin{figure}[ht]
\centering
\vskip -2.3pc
\definecolor{zzttqq}{rgb}{0.6,0.2,0.0}
\definecolor{qqqqff}{rgb}{0.0,0.0,1.0}

\begin{tikzpicture}[line cap=round,line join=round,>=triangle 45,x=0.8cm,y=0.8cm]
\clip(-3.7591033834920546,-1.1921850476563576) rectangle (4.1109064842149445,9.94786868158296);
\fill[color=zzttqq,fill=zzttqq,fill opacity=0.1] (-0.0,2.0) -- (3.51635305484543,6.145855962193699) -- (3.51635305484543,7.804347107277081) -- (-1.850273952358349,6.364901962487731) -- (-1.850273952358349,5.113210532236122) -- cycle;

\draw [->] (-3.0,-1.0) -- (3.970091198311639,1.31119781284686);
\draw [->] (2.0,-1.0) -- (-2.0,1.0);

\draw (1,0.45) node[anchor=north west] {\tiny$1$};
\draw (2.1,0.82) node[anchor=north west] {\tiny$2$};
\draw (-1.02,0.45) node[anchor=north west] {\tiny$1$};


\draw [dotted] (1.216,0.392)-- (1.1752763090478346,4.438380308636833);
\draw [dotted] (0.44,0.8)-- (0.4436267138365634,4.804205106242469);
\draw [dotted] (0.4436267138365634,4.804205106242469)-- (1.1752763090478346,4.422474882653979);
\draw [dotted] (0.4400131307372575,0.8144974672787895)-- (1.2157495009564916,0.41689004246141903);

\draw (-0.35,0.03) node[anchor=north west] {\tiny$0$};
\draw (-0.43,1.3) node[anchor=north west] {\tiny$1$};
\draw (-0.43,2.3) node[anchor=north west] {\tiny$2$};
\draw (-0.43,3.3) node[anchor=north west] {\tiny$3$};
\draw (-0.43,4.3) node[anchor=north west] {\tiny$4$};
\draw (-0.43,5.3) node[anchor=north west] {\tiny$5$};
\draw (-0.43,6.24) node[anchor=north west] {\tiny$6$};
\draw (1.1,-0.73) node[anchor=north west] {\tiny$-2$};
\draw (-2.77,-0.72) node[anchor=north west] {\tiny$-2$};
\draw (-0.485,8.007746964692968) node[anchor=north west] {\tiny$d_1$};
\draw (3.1,1.2) node[anchor=north west] {\tiny$d_2$};
\draw (-1.89,0.87) node[anchor=north west] {\tiny$d_3$};


\draw (-2.5,9.5) node[anchor=north west] {\tiny$\mathbb H: d_1=d_2+d_3+d_4+d_5 \text{ with } d_4=d_5=1$};

\draw (1.4,6.82) node[anchor=north west] {\tiny$2p-3s=0$};

\draw [dotted] (-0.7780406853788455,0.38902034268942276)-- (-0.7706900937663351,4.424780245597737);
\draw [dotted] (0.0,4.0)-- (-0.7706900937663351,4.440426388475882);
\draw [dotted] (0.4436267138365633,4.804205106242469)-- (-0.7706900937663351,4.440426388475882);
\draw [dotted] (0.0,4.0)-- (1.1754094081369262,4.425155340059457);
\draw [dotted] (-0.7780406853788455,0.38902034268942276)-- (0.4400131307372575,0.8144974672787895);
\draw [dotted] (0.0,2.0)-- (1.588839569532499,-0.8217045670689656);
\draw [dotted] (0.0,2.0)-- (-2.349933259798524,-0.7832412519248264);
\draw (0.4436267138365633,4.804205106242469)-- (0.43406290785083973,8.44583896528103);

\draw [shift={(1.7730116733844135,5.479691340420367)},line width=1pt] plot[domain=-1.3985416675424274:1.4377523292568368,variable=\t]({1.0*0.40374823940187243*cos(\t r)+-0.0*0.40374823940187243*sin(\t r)},{0.0*0.40374823940187243*cos(\t r)+1.0*0.40374823940187243*sin(\t r)});
\draw [shift={(-6.058863462979698,41.211915357127204)},line width=1pt] plot[domain=4.853729933557558:4.927684404849162,variable=\t]({1.0*36.983830848250015*cos(\t r)+-0.0*36.983830848250015*sin(\t r)},{0.0*36.983830848250015*cos(\t r)+1.0*36.983830848250015*sin(\t r)});


\draw (-0.15,4.85) node[anchor=north west] {\tiny$(4,1,1,1,1)$};

\draw (-0.15,5.75) node[anchor=north west] {\tiny$(5,1,1,1,1)$};

\draw [shift={(-5.34139657895604,-40.899890696422965)},line width=1pt] plot[domain=1.4174212517259508:1.4757190732625398,variable=\t]({1.0*47.322058385480084*cos(\t r)+-0.0*47.322058385480084*sin(\t r)},{0.0*47.322058385480084*cos(\t r)+1.0*47.322058385480084*sin(\t r)});

\draw [->] (0.0,-0.5194009038961178) -- (0.0,9);

\draw [->] (2.1394924815686585,6.396194248244021) -- (1.2163700517580962,6.036332962046684);

\draw [->] (0.7000473367793069,9.05) -- (1.65,7.35);

\draw [color=black,line width=1pt] (-0.0,2.0)-- (3.51635305484543,6.145855962193699);
\draw [color=black,line width=1pt] (3.51635305484543,6.145855962193699)-- (3.51635305484543,7.804347107277081);
\draw [color=black,line width=1pt] (3.51635305484543,7.804347107277081)-- (-1.850273952358349,6.364901962487731);
\draw [color=black,line width=1pt] (-1.850273952358349,6.364901962487731)-- (-1.850273952358349,5.113210532236122);
\draw [color=black,line width=1pt] (-1.850273952358349,5.113210532236122)-- (-0.0,2.0);


\begin{scriptsize}
\draw [color=black] (0.44,5.7) circle (2pt);
\draw [color=black] (0.44,6.7) circle (2pt);
\draw [color=black] (0.44,7.7) circle (2pt);

\draw [fill=black] (0.44,4.804205106242469) circle (2pt);

\end{scriptsize}

\end{tikzpicture}
\vskip -1pc

\caption{}\label{FIG:r=5}

\end{figure}


\item $r \geq 6$: for every $\lambda$, we get that  $\sigma_2(\XX_{2,\lambda})$ is defective.
 
  \end {itemize}
 

 \section{ Further Remarks and Questions }

In light of what we have proved in this paper, the question that puzzles us the most is: What can one say about the dimensions of the higher secant varieties of $\XX_{2, \lambda}$?

As we have noted, up to now the only results in this direction are due to Abo \cite{A} who gave a complete answer to this question for $\lambda = [1,\ldots , 1]$.

It's clear that if $\sigma_2(\XX_{2,\lambda})$ fills its ambient space for some $\lambda$  then $\sigma_s(\XX_{2,\lambda})$ has the expected dimension for that $\lambda$ and $s\geq 2$.  Our next goal is to describe all the $\lambda$ such that $\sigma_2(\XX_{2,\lambda})$ fills its ambient space, in other words, if  $\lambda \vdash d$, the generic form $F$ in degree $d$ can be written as  $F = F_1+F_2$, where the $F_i$ belong to $ \XX_{2,\lambda}$.

\begin{prop}\label{sigma2fills}
$\sigma_2(\XX_{2,\lambda}) $ fills the ambient space $\PP^N$,  $(N=  {d+2\choose 2}-1)$,  if and only if  
either $3s-2p \geq  0$, or    $\lambda= [2,2,2,1]$.
\medskip

\noindent (Note: A complete description of those $\lambda$ for which $3s-2p\geq0$ is given in Proposition  \ref{2p>3s} i). )

\begin {proof} We continue with the notation we have used throughout the paper.  We know that
$$
\dim \sigma_2(\XX_{2,\lambda}) = \dim ( (I_{F})_d + (I_{G})_d) -1  =  2 {d+2\choose 2}-2D- \dim (I_Z)_d-1,
$$
hence
$$
\dim \sigma_2(\XX_{2,\lambda}) =N \Leftrightarrow 
\dim (I_Z)_d= {d+2\choose 2} -2D.
$$

Since
$$
\dim (I_Z)_d = \max \left \{{d+2\choose 2} - 2D\ ; 0 \right \} +\delta ,
$$
we get
$$
\dim \sigma_2(\XX_{2,\lambda}) = N \Leftrightarrow 
 {d+2\choose 2} -2D \geq 0 \hbox { and } \delta =0.
$$

For $r=2$, the conclusion follows from Lemma \ref {r=2}.

\medskip

Let $r>2$,  and $d_1 \geq s$. 

In this case  from Lemmata \ref {d1>=s-1,2p<=3s}  and \ref {d1>=s-1,2p>3s}   
  it follows that 
$\delta =0 \Leftrightarrow  3s-2p \geq  0$. 
Since 
$${d+2\choose 2} -2D=  {d_1-s+2 \choose 2} +3s-2p, \eqno (6)
$$
hence, for  $ 3s-2p \geq  0$,  we get
${d+2\choose 2} -2D>0.
$
It follows that for $d_1 \geq s$
$$
\dim \sigma_2(\XX_{2,\lambda}) = N \Leftrightarrow  3s-2p \geq  0 ,$$
and we are done for  $r>2$,  and $d_1 \geq s$. 

\medskip

Now let $r>2$,  and $d_1 < s$. 

In this case we always have  $\delta =0$ (see  Proposition \ref {d1<s}), so we have  to check when 
 $${d+2\choose 2} -2D \geq 0.
$$
Consider 
${d+2\choose 2} -2D$
as a function of $d_1$, say
$$\varphi (d_1) ={d+2\choose 2} -2D. \eqno (7)$$
We  study this function for $d_2 \leq d_1 \leq s-1$.
Now, the parabola represented by (7) has its minimum when $d_1= s- \frac 3 2$.  So  for $d_1 < s- \frac 3 2$, $\varphi$ is decreasing.
Moreover (see (6))  we have
$$ \varphi (s-1) =\varphi (s-2) = 3s-2p.$$
Hence
$3s-2p \geq 0$ implies  $ \varphi (d_1) \geq 0  .$

It remains to check that if  $3s-2p <0$, then  the only case for which    $d_1$   is such that  $ \varphi (d_1) =0$, and   $d_2 \leq d_1 <s $, is when $d_1=2 $ and $\lambda = [2,2,2,1]$.

The case $d_1 = d_2 = s-1$ turns out to be verified only by
the partition $[a; a; 1]$ that was treated in Lemma 4.5 and for which the statement
is true; so we can assume $d_2<s-1$.

Since the parabola of equation (7) is decreasing in the interval $[d_2, s-2] $  and  $\varphi (s-2)<0$,   if we prove that, except for $\lambda = [2,2,2,1]$,   $ \varphi (d_2) <0$, we are done.   

So now we  compute  $ \varphi (d_2) $ for $\lambda \neq [2,2,2,1]$.
In order to do that, it is useful to consider the following equality
$$ 2\varphi (d_1) =
 - \sum _{i=1}^r ((d_1+ \cdots +d_r-2d_i ) (d_i-1)) - (r-5) (d_1+ \cdots +d_r) +2, 
$$
which, for $d_1=d_2$, becomes
$$2 \varphi (d_2) =-2(d_3+ \cdots +d_r ) (d_2-1)
 - \sum _{i=3}^r ((2d_2+d_3 \cdots +d_r-2d_i ) (d_i-1)) $$
 $$- (r-5) (2d_2+d_3+ \cdots +d_r) +2.
$$
Recall that, since $3s-2p <0$,   we have only to consider the cases listed in Proposition \ref {2p>3s}, that is, 
\begin {itemize}
\item   $r =3$, $d_3=2$ and $d_2 \geq 7$;
\item   $r =3$, $d_3=3$ and $d_2 \geq 4$;
\item   $r =3$, $d_3 \geq 4$;
\item   $r =4$, $d_3 \geq 2$;
\item   $r =4$, $d_2 \geq 5$;
\item   $r =5$, $d_2 \geq 2$;
\item   $r \geq 6$.
\end {itemize}

For $r\geq 6$,  we get
$$\varphi (d_2)   \leq
 -  (d_1+ \cdots +d_r) +2 <0.
$$

For  $r = 5$,  since $d_2 \geq 2$, we have
$$ \varphi (d_2) \leq -2(d_3 \cdots +d_5 ) +2 <0.$$

For $r=4$ and $d_3 \geq 2$, and so also $d_2 \geq 2$,  we get
$$ \varphi (d_2) \leq-2d_4 
   - (2d_2+d_3 +d_4-2d_4 ) (d_4-1)+2.$$
  Hence, if $d_4 \geq 2$,  then $ \varphi (d_2) <0$. 
  
 For $r=4$,  $d_4=1$ and $\lambda \neq [2,2,2,1]$, we have $d_3 \geq 3$, hence
  $$ \varphi (d_2) \leq-4(d_3 +d_4 ) - 2(2d_2+d_3+d_4-2d_3 )  
 + (2d_2+d_3+d_4) +2 <0.
$$

It is an easy computation  to check that in  all the remaining cases we have  $\varphi (d_2) <0$.

\end {proof}

\end{prop}

Other questions which come to mind about the varieties $\XX_{2,\lambda}$ and their secant varieties are:

$a)$\ Which of these varieties is arithmetically Cohen-Macaulay (aCM)? (we think that they are all aCM)

$b)$\ Can one find equations for any of these varieties?  So far we know of no equations satisfied by any of them!

$c)$\ Assuming that the varieties are aCM, what is their Cohen-Macaulay type? (This asks about the rank of the last term in a finite free resolution of the defining ideal.  One can ask about all the ranks and all the graded Betti numbers as well!)

It certainly would be illuminating to have answers to these questions, even for $\lambda = [1, \ldots , 1]$.
 

\section*{Acknowledgments} 

The authors are grateful to the referee for her/his accurate work.

The first and third authors wish to thank Queen’s University, in the person of the second author,
for their kind hospitality during the preparation of this work. The first three  authors enjoyed support from
NSERC (Canada). 
The first author was also  supported by GNSAGA of INDAM and by MIUR funds (Italy).
The third author was also  supported  by MIUR funds (Italy).
The fourth author  was supported by Basic Science Research Program through the National Research Foundation of Korea (NRF) funded by the Ministry of Education, Science, and Technology (2013R1A1A2058240).

 
\bibliographystyle{alpha}
\bibliography{biblioSecant}


\end{document}